\documentclass[11pt]{article}
\usepackage{amssymb,amsmath,bm}
\usepackage[colorlinks=true, urlcolor=blue,linkcolor=blue, citecolor=blue]{hyperref}
\usepackage{mathrsfs,verbatim}
\usepackage{caption,float}
\usepackage{constants,color}
\usepackage{enumerate, tabu}
\usepackage{bbm}
\usepackage{appendix}
\usepackage{graphicx, caption,subcaption}
\begin{document}
\newcommand{\bea}{\begin{eqnarray}}
\newcommand{\ena}{\end{eqnarray}}
\newcommand{\beas}{\begin{eqnarray*}}
\newcommand{\enas}{\end{eqnarray*}}
\newcommand{\beq}{\begin{equation}}
\newcommand{\enq}{\end{equation}}
\def\qed{\hfill \mbox{\rule{0.5em}{0.5em}}}
\newcommand{\bbox}{\hfill $\Box$}
\newcommand{\ignore}[1]{}
\newcommand{\ignorex}[1]{#1}
\newcommand{\wtilde}[1]{\widetilde{#1}}
\newcommand{\mq}[1]{\mbox{#1}\quad}
\newcommand{\bs}[1]{\boldsymbol{#1}}
\newcommand{\qmq}[1]{\quad\mbox{#1}\quad}
\newcommand{\qm}[1]{\quad\mbox{#1}}
\newcommand{\nn}{\nonumber}
\newcommand{\Bvert}{\left\vert\vphantom{\frac{1}{1}}\right.}
\newcommand{\To}{\rightarrow}
\newcommand{\supp}{\mbox{supp}}
\newcommand{\law}{{\cal L}}
\newcommand{\Z}{\mathbb{Z}}

\newcommand{\ucolor}[1]{\textcolor{blue}{#1}}  
\newcommand{\ucomm}[1]{\marginpar{\tiny\ucolor{#1}}}  

\newcommand{\icolor}[1]{\textcolor{red}{#1}}  
\newcommand{\icomm}[1]{\marginpar{\tiny\icolor{#1}}}  

\newcommand{\ccolor}[1]{\textcolor{green}{#1}}  
\newcommand{\ccomm}[1]{\marginpar{\tiny\ccolor{#1}}}  

\newtheorem{theorem}{Theorem}[section]
\newtheorem{corollary}{Corollary}[section]
\newtheorem{conjecture}{Conjecture}[section]
\newtheorem{proposition}{Proposition}[section]
\newtheorem{lemma}{Lemma}[section]
\newtheorem{definition}{Definition}[section]
\newtheorem{example}{Example}[section]
\newtheorem{remark}{Remark}[section]
\newtheorem{case}{Case}[section]
\newtheorem{condition}{Condition}[section]
\newcommand{\pf}{\noindent {\it Proof:} }
\newcommand{\proof}{\noindent {\it Proof:} }
\frenchspacing


\title{\bf On Unfair Permutations}

\author{\.{I}lker Arslan\thanks{I\c{s}{\i}k University, Istanbul, Turkey. {\tt Email:\,ilkerarslan@sabanciuniv.edu}.} \hspace{0.2in} \"{U}m\.{i}t I\c{s}lak\thanks{Bo\u{g}azi\c{c}i University, Istanbul, Turkey. {\tt Email:\,umit.islak1@boun.edu.tr}.}  \hspace{0.2in} Cihan Pehlivan\thanks{At{\i}l{\i}m University, Istanbul, Turkey. {\tt Email:\,cpehlivan@ku.edu.tr}.}} \vspace{0.25in}

\maketitle

\begin{abstract} 

In this paper we study the inverse of so-called unfair permutations. Our investigation begins with comparing this class of permutations with uniformly random permutations, and showing that they behave very much alike in case of  locally dependent random variables. As an example of a globally dependent statistic we use the number of inversions, and show that this statistic satisfies a central limit theorem after proper centering and scaling.
\end{abstract}

{\bf Keywords:} Random permutations, uniform permutations, descents, inversions, Stein's method, size biased coupling.

\section{Introduction}\label{sec:intro}

Letting $X_1,\ldots,X_n$ be independent and identically distributed (i.i.d.) random variables each of which is uniformly distributed over $(0,1)$,  and $R_1,\ldots,R_n$ be the corresponding ranks, it is well-known that the distribution of $(R_1,\ldots,R_n)$ is the same as the distribution of a uniformly distributed random permutation in $S_n$, the symmetric group on $n$ distinct letters. This result, often attributed to R\'enyi, of course remains true if we replace the uniform distribution over $(0,1)$ with any other continuous distribution. 

The purpose of this note is to study two related random permutation models; unfair permutations and their inverses. Former of these models was first introduced  in \cite{PSW:2016} by following the ensuing game description.\footnote{Let us note that the definition of an unfair permutation given in \cite{PSW:2016}   is not totally clear to us. At several points they use properties of inverses of unfair permutations as called in this paper, but then, for example, when they deal with the number of descents or the probability of a given permutation, their results are based on unfair permutations - again in the setting of this manuscript.}There are $n$ players labeled $1$ through $n$, and $i^{th}$  player chooses $i$ independent random numbers each of which is uniformly distributed over $(0,1)$, and picks the maximum, say $Z_i$, as her score. Then, the resulting unfair permutation  is $\gamma_n = (\gamma_n(1),\ldots,\gamma_n(n))$, where $\gamma_n(i)$ is the player whose rank is $i$, i.e. who has the $i^{th}$ smallest value. Here, $\gamma_n$ is unfair in the sense that when $i$ is large, $\gamma_n(i)$ favors to have larger values, and vice versa. The motivation of \cite{PSW:2016} for introducing unfair permutations is related to the theory of partitions, see the cited work for a relevant discussion. 

Given the same game description, we  define another random  permutation model by letting $\rho_n=(R_1,R_2,\ldots,R_n)$, where $R_i$ is the  rank of the $i^{th}$ player. Noting that   $\gamma_n(i)$ is the label of the player with  rank $i$ and that  $\gamma_n^{-1}(j)$ is the rank of  player $j$, we see that   $\rho_n=\gamma_n^{-1}$. For this reason, the permutation $\rho_n$ is said to have the   inverse-unfair permutation.

In order to clarify the definitions, let us continue with an example. Assume that there are 4 players, and player $i$, $i= 1,\ldots,4$, picks $i$ random numbers $\{X_j^{(i)}\}_{j=1}^i$  independently each of which is uniform over $(0,1)$, and that the resulting  random numbers turn out to be $\{X_j^{(1)}\}_{j=1}^1 = \{0.75\}$, $\{X_j^{(2)}\}_{j=1}^2 = \{0.15, 0.95\}$, $\{X_j^{(3)}\}_{j=1}^3 = \{0.12, 0.31, 0.72\}$, and $\{X_j^{(4)}\}_{j=1}^4 = \{0.03, 0.27, 0.34, 0.52\}$. Then, recalling that $Z_i = \max_{j =1,\ldots,i} \{X_j^{(i)}\}$, $$Z_1 = 0.75, \quad  Z_2 = 0.95, \quad Z_3 = 0.72, \quad \text{and} \quad Z_4 = 0.52,$$ and so the corresponding rank sequence is   $$R_1 = 3, \quad  R_2 = 4, \quad R_3 = 2, \quad \text{and} \quad R_4 = 1,$$ yielding an unfair permutation $$\gamma_4 = (4, 3, 1, 2).$$ The  corresponding  inverse-unfair permutation is just the rank sequence $$\rho_4 = (3,4,2,1).$$ Here, and below, the permutations are written in one-line notation. For example, $\gamma_4 = (4, 3, 1, 2) = 
  \bigl(\begin{smallmatrix}
    1 & 2 & 3 & 4\\
    4 & 3 & 1 & 2
  \end{smallmatrix}\bigr)$.
   
In next section we will start the discussion by comparing inverse-unfair/unfair permutations to uniform permutations in various ways. In general, we show that statistics of these two random permutation models can behave quite differently when the underlying dependence is global, but that this is not the case when underlying dependence is only local. Also, we will focus on some specific examples of both locally  and globally dependent statistics. Letting $\tau_n$ be a permutation in $S_n$, regarding local dependence, we  will analyze  the number of $m$-descents   when  $m$ is  a fixed natural number 
\begin{equation}\label{defn:descents}
D_{n,m}(\tau_n)=\#\{(i,j): 1\leq j-i\leq m, \tau_n(i)>\tau_n(j) \},
\end{equation}
and regarding globally dependent statistics we will focus on the number of inversions  defined by $$\textnormal{Inv}(\tau_n)  = \sum_{1 \leq i < j \leq n} \mathbf{1}(\tau_n(i) > \tau_n(j)).$$  

 Before moving to the main discussion, let us  fix some notation. First, $=_d$, $\rightarrow_d$ and
$\rightarrow_{\mathbb{P}}$ are used for equality in distribution,
convergence in distribution and convergence in probability,
respectively. $\mathcal{G}$ denotes a standard normal random
variable, and $C$ is used for constants (which may differ in each
line) that do not depend on any of the parameters. The notations  $d_W$ and $d_{TV} $ are reserved for the Wasserstein and total variation distances between probability measures, respectively.  For two
sequences $a_n$ and $b_n$, we write $a_n \sim b_n$  if  $\lim_{n
\rightarrow \infty} a_n / b_n =1$. Finally, $\lfloor \cdot \rfloor$ is used for the floor function.

Rest of the paper is organized as follows. Section \ref{sec:basics} contains some basic probability computations that will be required in following sections. Results of Section \ref{sec:comparison}  compare uniformly random permutations to  unfair permutations, and provides a general connection between the  two for locally dependent random variables. Same section   contains a central limit theorem for two locally dependent statistics. Later, in Section \ref{sec:inversions}, we also consider a globally dependent statistic, the number of inversions, and prove a central limit theorem.  We conclude the paper in Section \ref{sec:conclusion} with a discussion of two generalizations of  unfair permutations. 

\section{Basics}\label{sec:basics}

As in Introduction, consider $n$ players where player $i$ picks $i$ independent random numbers $\{X_j^{(i)}\}_{j=1}^i$ each of which is uniform over $(0,1)$. Set $Z_i = \max\{X_j^{(i)}: j =1,\ldots, i\}$, \footnote{From here on, we use the notation $Z_i$ for maximum of $i$ i.i.d. random numbers without further mention whenever it is clear from the context.} and $R_1,\ldots,R_n$ be the ranks of $Z_1,\ldots,Z_n$, respectively. Define the random permutations $\rho_n$ by setting $\rho_n = (R_1,R_2,\ldots,R_n)$ and $\gamma_n$ by setting $\gamma_n=(\gamma_n(1),\ldots,\gamma_n(n))$, where $\gamma_n(i)$ is the player whose rank is $i.$ 



As a starting point, let us note that there are certain statistics in $S_n$, say $T$ is one such example, so that $T(\rho_n) = T(\gamma_n)$  thanks to some sort of symmetry within $T$. For instance,  recalling that $\textnormal{Inv}(\tau_n)$  is the number of inversions in a permutation $\tau_n$, and noting that  $\textnormal{Inv}(\tau_n) = \textnormal{Inv}(\tau_n^{-1})$ for any $\tau_n$, we have 
\begin{equation}\label{eq:invprop}
\textnormal{Inv}(\rho_n) = \textnormal{Inv}(\gamma_n).
\end{equation}
A generalization to this observation   can be given by  considering the number of increasing (or, decreasing) subsequences  of a permutation $\tau_n$ of  a given length $m$, denoted by $\textnormal{Inc}_{n,m}(\tau_n)$. 
This is so since  if $\tau_n$ has an increasing subsequence indexed by $i_1<i_2<\cdots<i_m$, then $\tau_n^{-1}$ has an increasing subsequence indexed by $\tau_n (i_1) < \tau_n (i_2)<\cdots<\tau (i_m)$. 


Moving on to probabilistic considerations, we will now do some elementary observations  that will be used repeatedly throughout the paper. First,   letting $X_1,\ldots,X_i,Y_1,\ldots,Y_j$ be i.i.d. uniform random variables over $(0,1)$, we observe that 
  \begin{eqnarray}\label{usefulproperty}
 \nonumber \mathbb{P}(\rho_n(i) < \rho_n(j)) 
 &=& \mathbb{P}(\max \{X_1.\ldots,X_i \} < \max \{Y_1,\ldots,Y_j\})  \\ 
&=&   \sum_{\ell=1}^{j}    \mathbb{P}\left(\max\{X_1,\ldots,X_i,Y_1,\ldots,Y_j\} = Y_{\ell} \right) = \sum_{\ell=1}^{j}  \frac{1}{i+j} = \frac{j}{i+j},
\end{eqnarray}
where we use the i.i.d. assumption.
Following  the reasoning in derivation of \eqref{usefulproperty}, for $i_s\in \{1,\dots,n\}$ where $1\leq s\leq k$ and $i_{s_1}\neq i_{s_2}$ for $s_1\neq s_2$, one can also easily show that   \begin{equation}\label{usefulproperty2}\mathbb{P}(\rho_n(i_1)< \cdots < \rho_n(i_k)) = \prod_{\ell = 1 }^k  \left( \frac{i_{\ell}}{\sum_{j = 1}^{\ell} i_j} \right).\end{equation}
In particular, \eqref{usefulproperty2} yields $$\mathbb{P}(\rho_n = id) = \frac{2^n}{(n+1)!}, \quad \quad \text{and} \quad \quad \mathbb{P}(\rho_n = (n,n-1,\ldots,2,1)) = \frac{2^{n}n!}{(2n)!}.$$ Also, since the inverse of $(n,n-1,\ldots,2,1)$ is the same permutation, we have 
 $\mathbb{P}(\gamma_n = (n,n-1,\ldots,2,1)) = \frac{2^{n}n!}{(2n)!}.$
The following list provides the probability mass function of an inverse-unfair permutation in  $S_4$. 
 
\begin{table}[H]
\centering-
\begin{subtable}[t]{1.4in}
\begin{tabular}{ |cc| }
 \hline
$\rho_4$ & Probabilities  \\
 \hline
 $(1234)$ & 0.13333\\
\hline
$(1243)$ & 0.11428 \\
\hline
$(1324)$ & 0.1 \\
\hline
$(1342)$ & 0.06857\\
\hline
$(1423)$ & 0.075   \\
\hline
$(1432)$ & 0.06 \\
\hline
\end{tabular}
\end{subtable}
\quad
\begin{subtable}[t]{1.4in}
\begin{tabular}{|cc|}
\hline
$\rho_4$ & Probabilities  \\
\hline
$(2134)$ & 0.06666\\
\hline
$(2143)$ & 0.05714\\
\hline
$(2314)$ &  0.03333\\
\hline
$(2341)$ & 0.01714\\
\hline
$(2413)$ & 0.02857  \\
\hline
$(2431)$ & 0.015\\
\hline
\end{tabular}
\end{subtable}
\quad
\begin{subtable}[t]{1.4in}
\begin{tabular}{|cc|}
\hline
$\rho_4$ & Probabilities  \\
\hline
$(3124)$ & 0.04 \\
\hline
$(3142)$ & 0.025 \\
\hline
$(3214)$ & 0.02666\\
\hline
$(3241)$ & 0.01428\\
\hline
$(3412)$ & 0.01428\\
\hline
$(3421)$ & 0.01071\\
\hline
\end{tabular}
\end{subtable}
\quad
\begin{subtable}[t]{1.4in}
\begin{tabular}{|cc|}
\hline
$\rho_4$ & Probabilities  \\
\hline
$(4123)$ & 0.02666 \\
\hline
$(4132)$ & 0.02222 \\
\hline
$(4213)$ & 0.01777 \\
\hline
$(4231)$ & 0.01111\\
\hline
$(4312)$ & 0.01269 \\
\hline
$(4321)$ & 0.00952\\
\hline
\end{tabular}
\end{subtable}
\caption{Exact probabilities of all  inverse-unfair permutations in $S_4$.}
\end{table} 
Regarding the unfair case, for any permutation $(a_1,a_2,\ldots,a_n) \in S_n$,  it can be easily shown that \cite{PSW:2016}   
\begin{equation}\label{eq:unfairprobs}
\mathbb{P} (\gamma_n = (a_1,a_2,\ldots,a_n))  = \frac{n!}{\prod_{i=1}^n \sum_{j=1}^i a_j}.
\end{equation}
 The following list,  which looks very much like Table 1, provides all probabilities for unfair permutations in $S_4$. 
\begin{table}[H]
\centering
\begin{subtable}[t]{1.4in}
\begin{tabular}{ |cc| }
 \hline
$\rho_4$ & Probabilities  \\
 \hline
 $(1234)$ & 0.13333\\
\hline
$(1243)$ & 0.11428 \\
\hline
$(1324)$ & 0.1 \\
\hline
$(1342)$ & 0.075\\
\hline
$(1423)$ & 0.06857 \\
\hline
$(1432)$ & 0.06 \\
\hline
\end{tabular}
\end{subtable}
\quad
\begin{subtable}[t]{1.4in}
\begin{tabular}{|cc|}
\hline
$\rho_4$ & Probabilities  \\
\hline
$(2134)$ & 0.06666\\
\hline
$(2143)$ & 0.05714\\
\hline
$(2314)$ & 0.04\\
\hline
$(2341)$ & 0.02666\\
\hline
$(2413)$ & 0.02857\\
\hline
$(2431)$ & 0.02222\\
\hline
\end{tabular}
\end{subtable}
\quad
\begin{subtable}[t]{1.4in}
\begin{tabular}{|cc|}
\hline
$\rho_4$ & Probabilities  \\
\hline
$(3124)$ & 0.03333\\
\hline
$(3142)$ & 0.025\\
\hline
$(3214)$ & 0.02666\\
\hline
$(3241)$ & 0.01777\\
\hline
$(3412)$ & 0.01428\\
\hline
$(3421)$ & 0.01269\\
\hline
\end{tabular}
\end{subtable}
\quad
\begin{subtable}[t]{1.4in}
\begin{tabular}{|cc|}
\hline
$\rho_4$ & Probabilities  \\
\hline
$(4123)$ & 0.01714\\
\hline
$(4132)$ & 0.015\\
\hline
$(4213)$ & 0.01428\\
\hline
$(4231)$ & 0.01111\\
\hline
$(4312)$ & 0.01071\\
\hline
$(4321)$ & 0.00952\\
\hline
\end{tabular}
\end{subtable}
\caption{Exact probabilities of all unfair permutations in $S_4$.}
\end{table}
An induction argument can be used to show that the maximum of $\mathbb{P}(\rho_n = (a_1,a_2,\ldots,a_n))$ is attained at the identity permutation $\rho_n = id = (1,2,\ldots,n)$,  and the minimum is attained at $\rho_n = (n,n-1,\ldots,1)$. The same result holds for inverse-unfair permutations as well.

\section{Comparison to uniform permutations}\label{sec:comparison}

The purpose of this section is to  discuss  similarities and differences between uniform and unfair/inverse-unfair permutations.  
First, 
 we would like to see how far  are unfair permutations from the uniform ones. For this purpose,   recall that the total variation distance between two probability measures $\mu$ and $\nu$ is defined by $d_{TV}(\mu,\nu) = \sup_{A \subset \mathbb{R}} |\mu (A) - \nu(A) |.$ When the sample space $\mathcal{S}$ of  $\mu$ and $\nu$ is discrete, it is well known \cite{LPW:2009} that we may write $d_{TV}(\mu,\nu) = \frac{1}{2} \sum_{x \in \mathcal{S}} |\mu(x) - \nu(x)|.
$

\begin{theorem}\label{thm:TVgoestoone} 
 Let $\rho_n$ and $\pi_n$ be random permutations in $S_n$ with  inverse-unfair and uniform distributions, respectively. Then $$\lim_{n \rightarrow \infty} d_{TV}(\mathcal{L}(\rho_n),\mathcal{L}(\pi_n)) = 1,$$where $\mathcal{L}(\rho_n)$ and $\mathcal{L}(\pi_n)$ are the laws of $\rho_n$ and $\pi_n$.  Same result holds if we replace $\rho_n$ by an unfair permutation $\gamma_n$.

 \begin{proof}
 Define $A_n := \{\tau \in S_n : \tau(1) < \tau(n),\cdots, \tau(\log n) < \tau(n)\}$. Then, evidently, $$\mathbb{P}(\pi_n \in A_n) = \frac{1}{\log n + 1}.$$ (Here and below, $\log n$ is understood to be $\lfloor \log n\rfloor$, where $\lfloor \cdot \rfloor$ is the floor  function.) Next, letting $\{X_i\}_{i \geq 1}$ and $\{Y_i\}_{i \geq 1}$ be sequences of i.i.d. uniformly random variables over $(0,1)$, we have 
 \begin{eqnarray*}
\mathbb{P}(\rho_n \in A_n) &=&  \mathbb{P}\left(\max \left\{X_1,\ldots, X_{\frac{ (\log n) (\log n + 1) }{2} } \right\} < \max \{Y_1, \ldots ,Y_n \} \right) \\  &=&	  \frac{n}{n + \frac{ (\log n) (\log n + 1) }{2}} \longrightarrow 1.\end{eqnarray*} 
 So, we conclude   
\begin{eqnarray*}
 \liminf_{n \rightarrow \infty} d_{TV}(\mathcal{L}(\rho_n),\mathcal{L}(\pi_n)) &\geq& \liminf_{n \rightarrow \infty}\left( \mathbb{P}(\rho_n \in A_n) - \mathbb{P}(\pi_n \in A_n) \right) \\ &=& \lim_{n\rightarrow \infty} \left( \frac{n}{n + \frac{ (\log n) (\log n + 1) }{2}} - \frac{1}{\log n + 1} \right) = 1,
\end{eqnarray*} 
 as $n \rightarrow \infty$, proving the first claim. 

For the second claim, we just observe that  
\begin{eqnarray*}
d_{TV}(\mathcal{L}(\rho_n),\mathcal{L}(\pi_n))&=&\frac{1}{2}\sum\limits_{\sigma\in S_n}|\mathbb{P}\{\rho_n=\sigma\}-\mathbb{P}\{\pi_n=\sigma\}|\\
&=&\frac{1}{2}\sum\limits_{\sigma\in S_n}|\mathbb{P}\{\rho_n=\sigma^{-1}\}-\mathbb{P}\{\pi_n=\sigma^{-1}\}|\\
&=&\frac{1}{2}\sum\limits_{\sigma\in S_n}|\mathbb{P}\{\gamma_n=\sigma\}-\mathbb{P}\{\pi_n=\sigma^{-1}\}|\\
&=&\frac{1}{2}\sum\limits_{\sigma\in S_n}|\mathbb{P}\{\gamma_n=\sigma\}-\mathbb{P}\{\pi_n=\sigma\}|=d_{TV}(\mathcal{L}(\gamma_n),\mathcal{L}(\pi_n)),
\end{eqnarray*}
where we used the fact that $\pi_n$ is uniform. 
\hfill $\square$
 \end{proof}
 \end{theorem} 
  
The next result shows  that the moments of certain locally dependent statistics in uniform and inverse-unfair permutations behave similarly asymptotically.

\begin{theorem}\label{cor:momentequivalance} 
Let $\rho_n$ and $\pi_n$ be random permutations in $S_n$ with inverse-unfair and uniform distributions, respectively. For a given $\tau_n \in S_n$, let $$Y(\tau_n) = \sum_{i=1}^n \chi_i(\tau_n),$$ where $\chi_i(\tau_n)$ is a  function of the form $$\chi_i(\tau_n)  \\=   \mathbf{1}(\tau_n(i - m_{2,n}(i)) \triangle_{-m_{2,n}(i)}\cdots \triangle_{-2} \tau_n(i-1)  \triangle_{-1} \tau_n(i)  \triangle_0 \tau_n(i+1)  \triangle_{1}\cdots  \triangle_{m_{1,n}(i) - 1}\tau_n(i+m_{1,n}(i))),$$  with each $ \triangle_{j} \in \{<, > \}$.  Here, for $i \in \mathbb{N}$, $m_{1,n}(i) = \min\{i-1, r_1(i)\}$ and $m_{2,n}(i) = \min\{i-1,r_2(i)\}$ are  integer valued functions with $1 \leq r_1(i), r_2(i) \leq M $ for some constant $M < \infty$. Then,  for any $k \geq 1$, we have  $$\frac{\mathbb{E}\left[\left( Y(\rho_n)\right)^k \right]}{\mathbb{E}\left[\left(   Y(\pi_n)\right)^k \right]} \longrightarrow 1, \quad \text{as} \; \; n \rightarrow \infty.$$ 
\end{theorem}  

Let us demonstrate the use of the theorem with some well known locally dependent statistics in literature: 

\begin{table}[H]
\centering
\label{my-label}
\begin{tabular}{|l|l|l|l|}
\hline
 Statistic & Order of Expectation  & Order of Variance \\ \hline
$D_n(\rho_n)$ &  $n / 2$ & $n / 12$   \\ \hline 
 $M_n(\rho_n)$ & $n / 3$  & $2n / 45 $  \\ \hline
$LA_n(\rho_n)$ & $2n / 3$ & $8n / 45$    \\ \hline
$R_{n,m}(\rho_n)$ & $\frac{n}{m!}$ & $n \left(\frac{1}{m!} \left(1 - \frac{2m -1}{m!} \right)  + 2 \sum_{k=1}^{m-1} \frac{1}{(m+k)!}\right)$   \\ \hline
\end{tabular}
\end{table}
Here the statistics are respectively the number of descents, the number of local maximums, the length of the longest alternating subsequence and the number of rising sequences of length $m$. Asymptotic orders   follow from Theorem \ref{cor:momentequivalance}, and the corresponding asymptotic results for uniform permutations  in \cite{Fulman:2004}, \cite{Romik:2011},  \cite{HR:2010}, and \cite{GG:2011}, respectively.

\emph{Proof of Theorem \ref{cor:momentequivalance}:} Let us first discuss the case $k=1$. Observe that  \begin{equation}\label{eqn:Expectratio1}
\frac{\mathbb{E}\left[Y(\rho_n) \right]}{\mathbb{E}\left[  Y(\pi_n) \right]} = \frac{\sum_{i=1}^n \mathbb{P}(\chi_i(\rho_n) = 1)}{\sum_{i=1}^n \mathbb{P}(\chi_i(\pi_n) = 1)}.
\end{equation} 
Following the notation in Introduction, for $j=1,\ldots,n$, let $\{X_i^{(j)}\}_{i=1}^j$  be the independent $U(0,1)$ random numbers used to form an inverse-unfair permutation. Also set $Z_j = \max \{X_i^{(j)}: i =1,\ldots,j\}$, and  define the events $$B_i^{(j)} = \{Z_j \in \{X_1^{(j)}, \ldots, X_{i - m_{2,n}(i)}^{(j)}\}\},$$ and $$A_i = \bigcap_{j=i-m_{2,n}(i)}^{i+m_{1,n}(i)} B_i^{(j)}.$$ Note that conditional on $A_i$, each of the random variables $Z_j$, $j=i - m_{2,n}(i),\ldots,i+m_{1,n}(i)$, can be considered the as the maximum of $i - m_{2,n}(i)$  i.i.d. $U(0,1)$ random variables so that $Z_{i-m_{2,n}(i)}, \ldots, Z_{i+m_{1,n}}(i)$ are i.i.d. as well. Since  $\{\chi_{i}(\rho_n) = 1\}$ is an event related to the relative ordering of the $Z_j$'s the conditional probability given that $A_i$ occurs should be the same as that of $\{\chi_i(\pi_n) = 1\}$.

Then $$\mathbb{P}(B_i^{(j)})  \geq \frac{i-m_{2,n}(i)}{j}, \; \text{for each} \; j=i-m_{2,n}(i),\ldots,i+m_{1,n}(i),$$ and so using independence$$\mathbb{P}(A_i) = \mathbb{P}\left( \bigcap_{j=i-m_{2,n}(i)}^{i+m_{1,n}(i)} B_i^{(j)} \right) \geq \prod_{j=i-m_{2,n}(i)}^{i+m_{1,n}(i)} \frac{i-m_{2,n}(i)}{j} \geq \left( \frac{i -M}{i+M}\right)^{m_{1,n}(i) + m_{2,n}(i) +1} \geq \left( \frac{i -M}{i+M}\right)^{2M + 1}.$$ This also yields $$\mathbb{P}(A_i^c) \leq 1 - \left( \frac{i -M}{i+M}\right)^{2M + 1}.$$ 
Going back to \eqref{eqn:Expectratio1}, we rewrite the right-hand as $$\frac{\sum_{i=1}^n \mathbb{P}(\chi_i(\rho_n) = 1)}{\sum_{i=1}^n \mathbb{P}(\chi_i(\pi_n) = 1)} = \frac{\sum_{i=1}^n \mathbb{P}(\chi_i(\rho_n) = 1 \big| A_i) \mathbb{P}(A_i)}{\sum_{i=1}^n \mathbb{P}(\chi_i(\pi_n) = 1)} +  \frac{\sum_{i=1}^n \mathbb{P}(\chi_i(\rho_n) = 1 \big| A_i^c) \mathbb{P}(A_i^c)}{\sum_{i=1}^n \mathbb{P}(\chi_i(\pi_n) = 1)}.$$ We will next prove
\begin{itemize}
\item[i.] $\frac{\sum_{i=1}^n \mathbb{P}(\chi_i(\rho_n) = 1 \big| A_i) \mathbb{P}(A_i)}{\sum_{i=1}^n \mathbb{P}(\chi_i(\pi_n) = 1)} \longrightarrow 1$, as $n\rightarrow \infty$,
\item[ii.] $\frac{\sum_{i=1}^n \mathbb{P}(\chi_i(\rho_n) = 1 \big| A_i^c) \mathbb{P}(A_i^c)}{\sum_{i=1}^n \mathbb{P}(\chi_i(\pi_n) = 1)} \longrightarrow 0$, as $n\rightarrow \infty$.
\end{itemize}

\emph{Proof of i.} Let $\epsilon > 0$. Observe that we have $$\left(\frac{i-M}{i+M} \right)^{2M+1} > 1 - \frac{\epsilon}{2} \Longleftrightarrow i >  \frac{2M}{ 1 - \exp \left(\frac{\log (1 - \epsilon /2)}{2M+1} \right)} -M \geq M,$$ where the last inequality follows since $\log (1 - \epsilon /2) < 0$ and so $\exp \left(\frac{\log (1 - \epsilon /2) }{2M + 1} \right)$. Now, letting $M^* = \max \left\{M, \left(\frac{2M}{1 - \exp\left(\frac{\log (1-\epsilon/2)}{2 M +1} \right)} \right) - M \right\} =  \frac{2M}{1 - \exp\left(\frac{\log (1-\epsilon/2)}{2 M +1} \right)} - M$, we write  
$$\frac{\sum_{i=1}^n \mathbb{P}(\chi_i(\rho_n) = 1 \big| A_i) \mathbb{P}(A_i)}{\sum_{i=1}^n \mathbb{P}(\chi_i(\pi_n) = 1)}  = \frac{\sum_{i=1}^{M^*} \mathbb{P}(\chi_i(\rho_n) = 1 \big| A_i) \mathbb{P}(A_i)}{\sum_{i=1}^n \mathbb{P}(\chi_i(\pi_n) = 1)}  + \frac{\sum_{i=M^*+1}^n \mathbb{P}(\chi_i(\rho_n) = 1 \big| A_i) \mathbb{P}(A_i)}{\sum_{i=1}^n \mathbb{P}(\chi_i(\pi_n) = 1)}. $$  First term on the left-hand side clearly converges to zero. Also it can be checked easily that the second term satisfies $$\lim_{n \rightarrow \infty}  \frac{\sum_{i=M^*+1}^n \mathbb{P}(\chi_i(\rho_n) = 1 \big| A_i) \mathbb{P}(A_i)}{\sum_{i=1}^n \mathbb{P}(\chi_i(\pi_n) = 1)} \geq 1- \epsilon,$$ proving claim i..

\emph{Proof of ii.} We have $$\frac{\sum_{i=1}^n \mathbb{P}(\chi_i(\rho_n) = 1 \big| A_i^c) \mathbb{P}(A_i^c)}{\sum_{i=1}^n \mathbb{P}(\chi_i(\pi_n) = 1)} \leq \frac{\sum_{i=1}^n  \mathbb{P}(A_i^c)}{\sum_{i=1}^n \mathbb{P}(\chi_i(\pi_n) = 1)} \leq  \frac{\sum_{i=1}^n  \left( 1 - \left(1 - \frac{2M}{i+M} \right)^{2M+1}\right)}{\sum_{i=1}^n \mathbb{P}(\chi_i(\pi_n) = 1)}.$$
 Now observe that for $i > M$, Bernoulli's inequality yields $$\left(1 - \frac{2M}{i +M } \right)^{2M + 1} \geq 1 - (2 M + 1) \frac{2M }{ i + M}.$$ Therefore, when $i > M$, we obtain  $$1 - \left(1 - \frac{2M}{i +M } \right)^{2M + 1} \leq  1 - \left( 1 - (2 M + 1) \frac{2M }{ i + M}  \right) =  \frac{(2M + 1 ) 2M}{i +M},$$ and since $M$ is a constant independent of $n$, this gives  
 $$\sum_{i=1}^n  \left( 1 - \left(1 - \frac{2M}{i+M} \right)^{2M+1}\right) \leq C \log n.$$
 Noting that $\liminf_{n \rightarrow \infty } \frac{1}{n} \sum_{i=1}^n \mathbb{P}(\chi_i(\pi_n) = 1) \geq c$, for some $c \in \mathbb{R}^+$, we conclude that $$\frac{\sum_{i=1}^n \mathbb{P}(\chi_i(\rho_n) = 1 \big| A_i^c) \mathbb{P}(A_i^c)}{\sum_{i=1}^n \mathbb{P}(\chi_i(\pi_n) = 1)}  \longrightarrow 0, $$ showing that claim ii is true.  Combining all above, result follows for $k=1$.  

For more general case, note that $\left( Y(\rho_n)\right)^k$ can still be considered as a sum of indicator random variables that are locally dependent, and proof follows in a similar way.\hfill $\square$

\bigskip

We   conclude this section with a more detailed  analysis of a  locally dependent statistic, the number of generalized descents. Expectation of a   special  case of this statistic, the number of descents, in unfair permutations was previously studied  in \cite{PSW:2016}, where they find its expectation and asymptotic growth. However, our treatment below   simplifies their computations significantly.  The theory for $D_{n,m}$ in  case of uniform permutations is well-established, see \cite{Fulman:2004} and \cite{Pike:2011} for relevant work.

For a given permutation $\tau_n \in S_n$, and for  $m\geq 1$, recall that  the number of $m$-descents in $\tau_n$ is defined by 
$D_{n,m}=\#\{(i,j): 1\leq j-i\leq m, \tau_n(i)>\tau_n(j) \}.$ When $m=1$, $D_{n,1}$ is known to be the number of descents in $\gamma_n$, and we simply write $D_n$ for $D_{n,1}$. In a similar way, the number of $m$-ascents in $\tau_n$ is defined   by
$A_{n,m}=\#\{(i,j): 1\leq j-i\leq m, \tau_n(i)< \tau_n(j) \},$ and $A_n := A_{n,1}$  is said to be the number of ascents in $\tau_n$.   
\begin{theorem}\label{thm:m-descents} 
Let $\rho_n$ be an inverse-unfair permutation in $S_n$ and $D_{n,m}$ be the number of $m$-descents in $\rho_n$. Then
\begin{equation*}\label{eqn:meandescents}
\mathbb{E}[D_{n,m}] = \frac{nm}{2}-\frac{m(m+1)}{4}-\sum_{k=1}^{m}\sum_{i=1}^{n-k} \frac{k}{2(2i+k)}, \quad \text{and} \quad \textnormal{Var}(D_{n,m}) \sim \frac{6n m + 4 m^3 + 3m^2 - m}{72}. 
\end{equation*} 
Further, \begin{equation*}\label{eqn:CLTlocal}
\frac{D_{n,m} - \mathbb{E}[D_{n,m}]}{\sqrt{\textnormal{Var}(D_{n,m})}} \longrightarrow_d \mathcal{G}, \qquad n \rightarrow \infty.
\end{equation*}
\end{theorem}
 
\begin{proof}
Noting that 
$$\mathbb{E}[D_{n,m}]=\sum_{k=1}^{m} \sum_{i=1}^{n-k}  \mathbb{P}(\rho_n(i) > \rho_n(i+k)),$$
and that 
$
\mathbb{P}(\rho_n(i)>\rho_n(i+k)) =\frac{i}{2i+k},$
we have 
\begin{align*}
\mathbb{E}[D_{n,m}]=\sum_{k=1}^{m} \sum_{i=1}^{n-k} \frac{i}{2i+k}=\sum_{k=1}^{m}\left[ \frac{n-k}{2}- \frac{1}{2}\sum_{i=1}^{n-k} \frac{k}{2i+k}\right]&=\frac{nm}{2}-\frac{m(m+1)}{4}-\sum_{k=1}^{m}\sum_{i=1}^{n-k} \frac{k}{2(2i+k)}.
\end{align*}
The asymptotics of $Var(D_{n,m})$ is an immediate consequence of Theorem \ref{cor:momentequivalance}, and results of \cite{Pike:2011} on the number of generalized descents in uniformly random permutations.  The  central limit theorem   follows from well-known asymptotic results on $m$-dependent sequences. Indeed, it is standard that besides the central limit theorem, one may obtain a convergence rate of order $1 / \sqrt{n}$ with respect to Kolmogorov distance \cite{CS}. 
\hfill $\square$
\end{proof}

It is easy to see that choosing $m=1$ in Theorem \ref{thm:m-descents} yields $\mathbb{E}[D_n]=\frac{n}{2}-\frac{\log n}{4}+O(1).$ Noting that $A_n+D_n=n-1$ we obtain $\mathbb{E}[A_n]=\frac{n}{2}+\frac{3\log n}{4}+O(1).$

\begin{remark} 
Following steps similar to the computation of $\mathbb{E}[D_{n,m}]$, we may actually compute exact values of higher moments of $D_{n,m}$. As an example, let us compute  $\textnormal{Var}(D_n)$.  Letting $U_i=\mathbf{1}(\rho_n(i)>\rho_n(i+1))$,   we have $\textnormal{Var}(U_i)=\frac{i(i+1)}{(2i+1)^2}$ since $\mathbb{E}[U_i]=\mathbb{P}(\rho_n(i)>\rho_n(i+1))=\frac{i}{2i+1}.$
When $i+1<j$, we have $\textnormal{Cov}(U_i,U_j)=0$ because $U_i$ and $U_j$ are independent. Also if $i+1=j$, then 
\begin{align*}
\mathbb{E}[U_iU_j]&=\mathbb{E}[U_iU_{i+1}]=\mathbb{P}(\rho_n(i)>\rho_n(i+1)>\rho_n(i+2))=\frac{i}{6i+9},
\end{align*}
and so
$$\textnormal{Cov}(U_i,U_{i+1})=\frac{i}{6i+9}-\frac{i}{2i+1}\frac{i+1}{2i+3}=\frac{-i(i+2)}{3(2i+3)(2i+1)}.$$
Hence
$$
\textnormal{Var}(D_n)=\sum_{i=1}^{n-1}\frac{i(i+1)}{(2i+1)^2}\ - \frac{2}{3}\sum_{i=1}^{n-1}\frac{i(i+2)}{(2i+3)(2i+1)}
= \frac{n}{12} + O(1). 
$$ 
\end{remark}
  
\section{Number of  inversions}\label{sec:inversions}
The number of inversions in a permutation $\tau \in S_n$ is defined by  $$\textnormal{Inv}(\tau) = \sum_{1 \leq i < j \leq n}  \mathbf{1} (\tau(i) > \tau(j)). $$  This is the number of pairs $(i,j)$ whose corresponding values are out of order.  The number of anti-inversions is defined in a similar way by setting $$\textnormal{AInv}(\tau) = \sum_{i < j}  \mathbf{1} (\tau(i) < \tau(j)).$$ Asymptotic properties of the number of inversions when $\tau$ is a uniformly random permutation are well studied. See   \cite{Fulman:2004} and \cite{Pike:2011}.

The number of anti-inversions in an unfair permutation $\gamma_n$ in $S_n$ was previously studied in \cite{PSW:2016}, where they proved that $$\mathbb{E}[\textnormal{AInv}(\gamma_n)] = \left(\frac{\log 2 }{2}\right) n^2 + O(n), \; \;  \text{and} \; \; \textnormal{Var}(\textnormal{AInv}(\gamma_n)) = \left(\frac{1}{3}  - \frac{\pi^2}{18} + \frac{2 \log 2}{3} - \frac{\log 3}{2} + \frac{2 \log^2 2}{3} \right) n^3 + o (n^3).$$

Clearly, these two imply that 
\begin{equation}\label{expectation:inversions}
\mathbb{E}[\textnormal{Inv}(\gamma_n)] = \left(\frac{ 1-  \log 2 }{ 2} \right) n^2  + O(n), 
\end{equation}
 and 
\begin{equation}\label{variance:inversions}
\textnormal{Var}(\textnormal{Inv}(\gamma_n)) = \left(\frac{1}{3}  - \frac{\pi^2}{18} + \frac{2 \log 2}{3} - \frac{\log 3}{2} + \frac{2 \log^2 2}{3} \right) n^3 + o (n^3).
\end{equation}
The main result of this section is the following  which provides a central limit theorem for the number of inversions in an inverse-unfair permutation setting. The same result also holds for standard unfair permutations after a small modification, see Remark \ref{rmk:CLTinvfair}.
\begin{theorem}\label{thm:CLTinversions}
Let $\rho_n$ be an inverse-unfair permutation in $S_n$. Then we have $$d_W\left(\frac{\textnormal{Inv}(\rho_n) - (1 - \log 2 / 2) n^2}{\textnormal{Var}(\textnormal{Inv}(\gamma_n))},\mathcal{G}\right) \leq \frac{C}{\sqrt{n}},$$ where $C$ is a constant independent of $n$.  In particular,  $$\frac{\textnormal{Inv}(\rho_n) - (1 - \log 2 / 2) n^2}{\sqrt{\frac{1}{3}  - \frac{\pi^2}{18} + \frac{2 \log 2}{3} - \frac{\log 3}{2} + \frac{2 \log^2 2}{3} }n^{3/2} } \longrightarrow_d \mathcal{G},$$ as $n \rightarrow \infty$. 
\end{theorem}

\begin{remark}\label{rmk:CLTinvfair}
For symmetry reasons, we also have  $\frac{\textnormal{Inv}(\gamma_n) - (1 - \log 2 / 2) n^2}{\sqrt{\frac{1}{3}  - \frac{\pi^2}{18} + \frac{2 \log 2}{3} - \frac{\log 3}{2} + \frac{2 \log^2 2}{3} }n^{3/2} } \longrightarrow_d \mathcal{G},$ as $n \rightarrow \infty$, where $\gamma_n$ is an unfair permutation. \end{remark}

\begin{remark}
The discussion on number of inversions can be generalized to increasing (or decreasing) sequences of arbitrary length. This statistic in uniformly random permutation framework was previously studied in \cite{IO:2016}.  Their proof is a lot simpler   due to underlying symmetry. 
\end{remark}

\subsection{Proof of Theorem \ref{thm:CLTinversions}} 

The proof will require size biased couplings from the Stein's method literature. In general, this method refers to a general technique  to provide estimation errors
for distributional approximations.  For a survey of the techniques from Stein's
method, see \cite{Ross:2011}.

Letting  $W$ be
a nonnegative and integrable random variable, the distribution of
$W^s$ is said to be $W$-size biased if we have
$$\mathbb{E}[W f(W)]=\mathbb{E}[W] \mathbb{E}[f(W^s)],$$ for all
functions $f$ for  which the   expectations exist and for which $\mathbb{E} |W f(W) | < \infty$. The main result we will need is the following theorem.
\begin{theorem}\label{thm:Steinsizebias} \cite{Ross:2011} Let $W \geq 0$ be a random variable with $\mathbb{E}[W] =
\mu $ and $\textnormal{Var}(W) = \sigma^2< \infty.$ Let $W^s$ be defined on the
same space as $W$ and have the size bias distribution with respect
to $W$. Then $$d_{W} \left(\frac{W-\mu}{\sigma}, \mathcal{G} \right) \leq
\frac{\mu}{ \sigma^2} \sqrt{\frac{2}{\pi}}
\sqrt{\textnormal{Var}(\mathbb{E}(W^s-W|W))} + \frac{\mu}{\sigma^3}
\mathbb{E}[(W^s-W)^2].$$ 
\end{theorem}
Clearly,  in the following, for a given nonnegative random variable $W$, we will need to construct a size biased coupling of $W$ with certain properties. The following result for sums of Bernoulli random variables will suffice for our purposes.  
 
\begin{proposition}
Let $X_1,\ldots,X_n$ be zero-one random variables with
$\mathbb{P}(X_i=1)=p_i.$ For each $i=1,\ldots,n$, let $(X_j^{i})_{j
\neq i}$ have the distribution of $(X_j)_{j\neq i}$ conditional on
$X_i=1.$ If $W= \sum_{i=1}^n X_i,$ $\mu= \mathbb{E}[W]$, and $I$ is
chosen independent of all else with $\mathbb{P}(I=i)=p_i / \mu,$
then $W^s =\sum_{j \neq I} X_j^I+1$ has the size-bias distribution
of $W.$

\end{proposition}

The proof is standard and we skip it referring to \cite{Ross:2011}.

\bigskip
 
\textbf{\emph{Proof of Theorem \ref{thm:CLTinversions}.}}  
First note that  $$\textnormal{Inv}(\rho_n) =_d \sum_{i < j} \mathbf{1}(Z_i > Z_j)$$ where $Z_i$ is the maximum of  $i$ independent $U(0,1)$ random variables, and where $Z_i$'s are independent. Let $$Y_{ij}=\mathbf{1}(Z_i > Z_j)$$ and $$W=\sum_{i < j} Y_{ij}.$$  
To size bias $W$, first denoting $p_{i,j} = \frac{\mathbb{E}[Y_{ij}]}{\sum_{k<l} \mathbb{E}[Y_{kl}]}$'s, we let $I$ be a  random variable taking values in $\{(i,j)\in \{1,\ldots,n\}: i < j\}$ with distribution $\mathbb{P}(I = (i,j)) = p_{i,j}$. Now  
if $Y_{i,j} =1$, then we keep the $Z_k$'s as they are. Otherwise, we sample  $(Z_i^*,Z_j^*)$ according to the  distribution of $(Z_i,Z_j)$ conditionally on $Z_i > Z_j$. Letting 
\begin{eqnarray*}
W^{ij}&=&\sum_{\{k,l\} \cap \{i,j\} = \emptyset } \mathbf{1}(Z_k >Z_l)+\sum_{s=1}^{i-1} \mathbf{1}(Z_s >Z_i^*)+\sum_{s=i+1, s \neq j}^{n} \mathbf{1}(Z_i^*>Z_s ) \\ 
&& + \sum_{s=1, s\neq i}^{j-1} \mathbf{1}(Z_s >Z_j^*)+\sum_{s=j+1}^{n} \mathbf{1}(Z_j^* >Z_s) +1,
\end{eqnarray*}
$W^I$ has $W$ size biased distribution.  
Also, for any $(i,j)$ with $1 \leq i < j \leq n$, \begin{eqnarray*}
                                               W^{ij}-W &=& \sum_{s=1}^{i-1} (\mathbf{1}(Z_s >Z_i^*)-\mathbf{1}(Z_s > Z_i)) + \sum_{s=i+1, s\neq j}^{n} (\mathbf{1}(Z_i^*> Z_s)-\mathbf{1}(Z_i > Z_s))\\
                                                 &+& \sum_{s=1, s \neq i}^{j-1} (\mathbf{1}(Z_s >Z_j^*)-\mathbf{1}(Z_s >Z_j)) + \sum_{s=j+1}^{n} (\mathbf{1}(Z_j^*>Z_s)-\mathbf{1}(Z_j >Z_s)) + 1 - \mathbf{1}(Z_i > Z_j).
                                             \end{eqnarray*}
This immediately gives $|W^{ij}-W| \leq 2n$ and so $\mathbb{E}|W^I-W|^2 \leq 4n^2.$ Next we focus on the  variance
estimate. Setting $\mathbf{Z} = (Z_1,\ldots,Z_n)$ observe  that 
$$\textnormal{Var}(\mathbb{E}[W^I-W|W]) \leq \textnormal{Var}(\mathbb{E}[W^I-W|\mathbf{Z}]).$$ See Lemma 4 of  \cite{Pike:2011} for a justification  of this last step. 
We then have \begin{eqnarray}\label{eq:vardecomp}
\nonumber                                    \textnormal{Var}(\mathbb{E}[W^I-W|W]) &\leq& \textnormal{Var}\left(  \sum_{i < j}  \frac{\frac{i}{i+j}}{\sum_{k < l} \frac{k}{k + l}} \mathbb{E}[W^{ij}-W| \mathbf{Z}]\right) \\
                                     &=&  \left(\sum_{i < j}  \frac{ \left( \frac{i}{i+j} \right)^2}{\left(\sum_{k < l} \frac{k}{k + l}\right)^2} \textnormal{Var}\left(\mathbb{E}[W^{ij}-W| \mathbf{Z}]\right) \right)\\
\nonumber                                      && +  \sum_{i < j, i'<j', (i,j) \neq (i',j')}  \frac{ \left( \frac{i}{i+j} \right) \left( \frac{i'}{i'+j'} \right)  }{\left(\sum_{k < l} \frac{k}{k + l}\right)^2}   \textnormal{Cov}(\mathbb{E}[W^{ij}-W| \mathbf{Z}], \mathbb{E}[W^{i'j'}-W| \mathbf{Z}]).
                                 \end{eqnarray}
Before moving further, we give some elementary estimates we will need below. First, it is clear that $$\sum_{k < l} \frac{k}{k + l} \leq C_1 n^2,$$ for some $C_1 > 0$. Also, $$\sum_{k < l} \frac{k}{k + l} \geq \sum_{k < l, k \geq n/2} \frac{k}{2l} \geq \sum_{k < l, k \geq n/2} \frac{ n/2}{2 n} = \binom{\lfloor n/2 \rfloor + 1}{2} \frac{1}{4} \geq C_2 n^2,$$   for some $C_2 > 0$. In particular, we conclude that  $\sum_{k < l} \frac{k}{k + l}$ is of order $n^2$. Lastly, note that $$\sum_{k<l} \left( \frac{k}{k + l} \right)^2 \leq \sum_{k<l}  \frac{k}{k + l} \leq C_3 n^2,$$ for some $C_3 > 0$.


Now for the first term in \eqref{eq:vardecomp}, we have 
\begin{eqnarray*}
  \left(\sum_{i < j}  \frac{ \left( \frac{i}{i+j} \right)^2}{\left(\sum_{k < l} \frac{k}{k + l}\right)^2} \textnormal{Var}\left(\mathbb{E}[W^{ij}-W| \mathbf{Z}]\right) \right) &=&   
   \left(\frac{\sum_{i < j}   \left( \frac{i}{i+j} \right)^2}{\left(\sum_{k < l} \frac{k}{k + l}\right)^2}   \left( \textnormal{Var}(W^{ij} - W) - \mathbb{E}[\textnormal{Var}( W^{ij} - W | \mathbf{Z})]\right) \right) \\
  &\leq&    \left(\frac{\sum_{i < j}   \left( \frac{i}{i+j} \right)^2}{\left(\sum_{k < l} \frac{k}{k + l}\right)^2}     \textnormal{Var}(W^{ij} - W)  \right) \\
  &\leq&  \left(\frac{\sum_{i < j}   \left( \frac{i}{i+j} \right)^2}{\left(\sum_{k < l} \frac{k}{k + l}\right)^2}   \mathbb{E}[(W^{ij} - W)^2] \right) \\ 
  &\leq&   4n^2 \frac{\sum_{i < j}   \left( \frac{i}{i+j} \right)^2}{\left(\sum_{k < l} \frac{k}{k + l}\right)^2}     \\
  &\leq& 4n^2 \frac{C_3 n^2}{C_2^2 n^4} \\
  &\leq& C_4.
  \end{eqnarray*}
Next, we focus on the second term in \eqref{eq:vardecomp} involving covariances. Observe that since the expectations are conditional on  $\mathbf{Z}$,  for any given  $i,j$,   $\mathbb{E}[W^{ij} - W | \mathbf{Z}]$ will be dependent on $C n$ many other $\mathbb{E}[W^{i'j'} - W | \mathbf{Z}]$. Denote by $\mathcal{S}_{ij}$ the set of $i',j'$ pairs with $i'< j'$ for which  $\mathbb{E}[W^{ij} - W | \mathbf{Z}]$ depends on $\mathbb{E}[W^{i'j'} - W | \mathbf{Z}]$.  Also observe that each of  the terms $\mathbb{E}[W^{ij} - W | \mathbf{Z}]$ is a sum of $2 n -1$ random variables that are themselves bounded by $1$. Combining these observations, we get                   
\begin{eqnarray*}
 \sum_{i < j, i'<j', (i,j) \neq (i',j')}  \frac{ \left( \frac{i}{i+j} \right) \left( \frac{i'}{i'+j'} \right)  }{\left(\sum_{k < l} \frac{k}{k + l}\right)^2}   \textnormal{Cov}(\mathbb{E}[W^{ij}-W| \mathbf{Z}], \mathbb{E}[W^{i'j'}-W| \mathbf{Z}]) 
\\ 
  = \sum_{i< j} \sum_{(i',j') \in \mathcal{S}_{ij}} \frac{ \left( \frac{i}{i+j} \right) \left( \frac{i'}{i'+j'} \right)  }{\left(\sum_{k < l} \frac{k}{k + l}\right)^2}   \textnormal{Cov}(\mathbb{E}[W^{ij}-W| \mathbf{Z}], \mathbb{E}[W^{i'j'}-W| \mathbf{Z}])  \\
\leq Cn^2  \frac{1}{\left(\sum_{k < l} \frac{k}{k + l}\right)^2 } \sum_{i< j}  \frac{i}{i+j} \sum_{(i',j') \in \mathcal{S}_{ij}} \frac{i'}{i'+ j'} \leq  C n^2 \frac{1}{n^4} n^2 n = Cn. 
\end{eqnarray*}

  Combining our estimates, we therefore obtain $$\textnormal{Var}(\mathbb{E}[W^I-W|W])  \leq C n,$$   for any $n\geq 1$ where $C$ is a constant independent of $n$.   Recalling also that $$\mathbb{E}\left[\textnormal{Inv}(\rho_n) \right] \sim (1 - \log 2 / 2) n^2,$$ and $$\textnormal{Var}(\textnormal{Inv}(\rho_n)) \sim  \left(\frac{1}{3}  - \frac{\pi^2}{18} + \frac{2 \log 2}{3} - \frac{\log 3}{2} + \frac{2 \log^2 2}{3} \right) n^3,$$ 
result now follows from Theorem \ref{thm:Steinsizebias}.

The second claim follows from a straightforward application of Slutsky's theorem.   \hfill $\square$

\section{Concluding remarks}\label{sec:conclusion}
We conclude the paper by noting that  unfair permutations admit two  natural generalizations  worth studying. 

(1) In standard inverse-unfair permutation framework, $i^{th}$ player chooses $i$ many i.i.d. uniform numbers over $(0,1)$ and picks the maximum. What if the $i^{th}$ player chooses $\phi(i)$ random numbers for some  function $\phi$?   Clearly, when $\phi$ is identically equal to $1$  and $\phi$ is the identity function, this setting recovers the  uniformly random and  the inverse-unfair permutation cases, respectively.  Depending on growth rate of $\phi_n$, one will  have quite different behaviours for underlying statistics.
 
(2) A second generalization can be given by making use of Markov chains. Let $
\{\phi(i)\}_{i \geq 1}$ be a Markov chain starting at time $t = 1$, with state space $\mathcal{S} \subset \mathbb{Z}^+$, and transition matrix $\mathbf{P}$. Also assume that $\phi(1) = 1$. Then we can define a variation of inverse-unfair permutations by saying that the $i^{th}$ player draws $\phi(i)$ numbers and chooses the maximum of these. 
Clearly, if the state space is $\mathcal{S} = \mathbb{Z}^+$, and the transition probability matrix $\mathbf{P}$ is $\mathbf{P}_{i,j} = \mathbf{1}(j = i +1)$, then we recover the standard unfair permutations.  However, the model is far more general thanks to the flexibility in choice of $\mathbf{P}$. In particular, depending on properties of $\mathbf{P}$ such as transience, recurrence, etc., statistics of the resulting model will differ from the corresponding unfair permutation statistics significantly. 

 \bigskip 
 
 \textbf{Acknowledgements} The author C. P. has been supported by TUBITAK within the project 113F059  entitled ''The conjecture of Mazur-Tate-Teitelbaum, CM elliptic curves and applications" 
as a postdoctoral researcher at  Ko\c{c} University.  The author \"{U}. I. is supported by  the Scientific and Research Council of Turkey [TUBITAK-117C047]. Parts of this paper were completed at the Nesin Mathematics Village, the authors would like to thank Nesin Mathematics Village for their kind hospitality. Also,  we would like to thank   an anonymous referee who detected some errors in original manuscript, and  whose comments  improved the paper significantly.

\end{document}